\documentclass[11pt,reqno]{amsart}

\usepackage[utf8]{inputenc}

\usepackage[text={150mm,230mm},centering]{geometry}   
\usepackage[mathscr]{eucal}
\usepackage[all]{xy}
\usepackage{mathrsfs,xypic}
\usepackage{amsfonts,amsmath,amsthm,amssymb,bm}
\usepackage{latexsym,tabularx,graphicx,pict2e}
\usepackage{pdfpages}
\usepackage{color}

\usepackage{hyperref}

\newtheorem{theorem}{Theorem}[section]
\newtheorem{lemma}[theorem]{Lemma}

\newtheorem{corollary}[theorem]{Corollary}
\newtheorem{proposition}[theorem]{Proposition}

\theoremstyle{definition}
\newtheorem{example}[theorem]{Example}
\newtheorem{definition}[theorem]{Definition}
\newtheorem{definition-lemma}[theorem]{Definition-Lemma}
\newtheorem{definition-theorem}[theorem]{Definition-Theorem}

\newtheorem*{notation}{Notations}

\allowdisplaybreaks

\title{Derived Morita invariance of Calabi-Yau algebras}

\author{Sirui Yu}

\author{Jieheng Zeng}

\address{Yu and Zeng: School of Mathematics, Sichuan University, Chengdu, 
Sichuan Province, 610064 P.R. China}
\email{banaenoptera@163.com, 
zengjh662@163.com}
\date{}

\begin{document}

\begin{abstract}
We study the derived equivalence of Calabi-Yau algebras and
show that, for two derived Morita equivalent algebras, if one is Calabi-Yau, then so is the other.


\end{abstract}
\maketitle
\setcounter{tocdepth}{1}\tableofcontents

\section{Introduction}
 In 1989, J. Rickard \cite{R} defined derived equivalence of
two algebras and built the derived Morita theory. Later in \cite{JR}, he showed that for two
derived equivalent finite dimensional algebras, if one is a symmetric algebra, then so is the other.

Symmetric algebras are closely related to Calabi-Yau algebras, a notion introduced by Ginzburg in \cite{G}.
In fact, Van den Bergh showed in \cite{VdB}
that if a Calabi-Yau algebra is Koszul, then its Koszul dual algebra is symmetric.
It is then natural to ask whether for two
derived Morita equivalent algebras, not necessarily Koszul, if one is Calabi-Yau, so is the other.
It turns out that the answer is yes. The main result of this paper is as follows:

\begin{theorem}\label{MainThm}
Let $A$ and $B$ be two $k$-algebras, where $k$ is a field of characteristic zero. If $A$ is a $d$-Calabi-Yau algebra, where $d \in \mathbb{N}$, and is derived equivalent to $B$,
 then $B$ is also a $d$-Calabi-Yau algebra.
\end{theorem}
The rest of the paper is devoted to the proof of the above theorem. It is organized as follows:

In \S 2 we recall some notions of the derived category of an abelian category, the derived functor, the triangulated category and some of their basic properties.

 In \S 3, we first remind the derived Morita theory and tilting theory developed by Rickard (\cite{R,JR}), and then prove Theorem \ref{MainThm}.

Throughout the paper, $k$ is a field of characteristic zero, all algebras are unital $k$-algebras. All modules are left modules unless otherwise stated and complexes are cochain complexes.

\section{Derived categories }

 \ \ \ \ In this section we recall some basic notions of the derived category. They are mostly taken from \cite{SG,CW}.

Derived category was originally introduced by Verdier. Nowadays it plays an increasingly important role in the study of homological algebra, algebraic geometry and many other fields.

Suppose $\mathcal{A}$ is an abelian category, let $Kom(\mathcal{A})$ be the category of cochain complexes over $\mathcal{A}$ and $X^{\bullet}, Y^{\bullet}$ be objects in $Kom(\mathcal{A})$.
For two morphisms $f,g \in \mathrm{Hom}_{Kom(\mathcal{A})}(X^{\bullet}, Y^{\bullet})$,
we say $f$ and $g$ are homotopy equivalent if there is a morphism
$s \in \mathrm{Hom}_{Kom(\mathcal{A})}(X^{\bullet}, Y[-1]^{\bullet})$ such that
$\delta \circ s + s \circ d = f - g$, where $d$ is the differential of $X^{\bullet}$ and $\delta$ is the differential of $Y^{\bullet}$.
For a complex $K^{\bullet}$, we denote by $K[i]^{\bullet}$ the complex with components $K[i]^{j} = K^{i+j}$ and $(-1)^{i} d_{K^{\bullet}}$ the corresponding differential.

\begin{definition}[Homotopy category]
Let $\mathcal{A}$ be an abelian category. The {\it homotopy category} $\mathcal{K}(\mathcal{A})$ is defined as follows:
the objects in $ \mathcal{K}(\mathcal{A})$ are the objects in $Kom(\mathcal{A})$,
and
the morphisms in $\mathcal{K}(\mathcal{A})$
are the morphisms in $Kom(\mathcal{A})$ modulo homotopy equivalences.
\end{definition}


 \begin{definition}[Derived category]
Let $\mathcal{A}$ be an abelian category and $Kom(\mathcal{A})$ be the category of complexes over $\mathcal{A}$.
Then there exists a category $\mathcal{D}(\mathcal{A})$, called the {\it derived category} of $\mathcal A$,
and a functor $Q:Kom(\mathcal{A}) \rightarrow \mathcal{D}(\mathcal{A})$ such that:

(1) $Q(f)$ is an isomorphism for any quasi-isomorphism $f$, i.e., any morphism of complexes inducing isomorphism on cohomology.

(2) Any functor $F: Kom(\mathcal{A}) \longrightarrow \mathcal{D}$ that
sends quasi-isomorphisms to isomorphisms can be uniquely factored through
$\mathcal{D}(\mathcal{A})$, i.e., there exists a unique functor
$G: \mathcal{D}(\mathcal{A}) \longrightarrow \mathcal{D}$ such that the following diagram
$$
\xymatrixcolsep{4pc}
\xymatrix{
Kom(\mathcal{A}) \ar[d]_{Q} \ar[r]^{F} & \mathcal{D} \\
\mathcal{D}(\mathcal{A}) \ar@{.>}[ur]_{\exists !G} &
                        }
$$
commutes.
\end{definition}

Equivalently, the derived category $\mathcal{D}(\mathcal A)$ of $\mathcal{A}$
is the category whose objects are objects in $\mathcal{K}(\mathcal{A})$ and
whose morphisms are morphisms of $\mathcal{K}(\mathcal{A})$ localized at quasi-isomorphisms.
Let $\mathcal{A}$ be an abelian category. For any $M, N \in \mathcal{D}(\mathcal{A})$, if $f \in \mathrm{Hom}_{\mathcal{K}(\mathcal{A})}(M, N)$ is a quasi-isomorphism, we denote by $f^{-1} \in \mathrm{Hom}_{\mathcal{D}(\mathcal{A})}(N,M)$ its inverse. The localization of homotopy category by quasi-isomorphisms is canonically isomorphic to the derived category. Strictly speaking, we have the following proposition and definition.

\begin{proposition}[See \cite{SG}, \S3.4.4-3.4.5]
Let $\mathcal{A}$ be an abelian category. For any $M, N \in \mathcal{D}(\mathcal{A})$, and $g \in \mathrm{Hom}_{\mathcal{D}(\mathcal{A})}(M, N)$, there is a morphism $f \in \mathrm{Hom}_{\mathcal{K}(\mathcal{A})}(M, N)$ and a quasi-isomorphism $s \in \mathrm{Hom}_{\mathcal{K}(\mathcal{A})}(M, N)$ such that $g = f \circ s^{-1}$ in $\mathcal{D}(\mathcal{A})$. The same is true for $\mathcal{D}^{*}(\mathcal{A})$ and $\mathcal{K}^{*}(\mathcal{A})$, where $*=+,-,b,\text{or }\emptyset$.
\end{proposition}

\begin{definition}[Localization functor]
Let $\mathcal{A}$ be an abelian category. The {\it localization functor associated to $\mathcal{K}(\mathcal{A})$}, denoted by $Q_{\mathcal{A}}: \mathcal{K}(\mathcal{A}) \longrightarrow \mathcal{D}(\mathcal{A})$ is defined as follows:
For any $M \in \mathcal{K}(\mathcal{A})$, $Q_{\mathcal{A}}(M) = M$ in $\mathcal{D}(\mathcal{A})$,
and for any $M, N \in \mathcal{K}(\mathcal{A})$, the functor  \[Q_{\mathcal{A}}(M, N) : \mathrm{Hom}_{\mathcal{K}(\mathcal{A})}(M, N) \longrightarrow \mathrm{Hom}_{\mathcal{D}(\mathcal{A})}(Q_{\mathcal{A}}(M), Q_{\mathcal{A}}(N)),\] is given by
$$
 Q_{\mathcal{A}}(M, N)(h) = h \circ (id_{M})^{-1},
$$
for any $h \in \mathrm{Hom}_{\mathcal{K}(\mathcal{A})}(M, N)$.
\end{definition}

\begin{notation}
Let $\mathcal{A}$ be an abelian category, denote by $Kom^+(\mathcal{A})$ (resp. $Kom^-(\mathcal{A})$)
the subcategory of
$Kom(\mathcal A)$ such that for any $M$ in $Kom^+(\mathcal{A})$ (resp. $Kom^-(\mathcal{A})$),  $M^{i} = 0$ for $i \ll 0  $
(resp. $i\gg 0$),
and
$Kom^b(\mathcal{A})=Kom^+(\mathcal{A}) \bigcap Kom^-(\mathcal{A})$.

Similarly, the subcategories
$\mathcal{K}^+(\mathcal{A})$, $\mathcal{K}^-(\mathcal{A})$, $\mathcal{K}^b(\mathcal{A})$
and
$\mathcal{D}^+(\mathcal{A})$, $\mathcal{D}^-(\mathcal{A})$, $\mathcal{D}^b(\mathcal{A})$ in
$\mathcal K(\mathcal A)$ and $\mathcal D(\mathcal A)$ respectively are defined similarly.
 \end{notation}

\begin{example}
For any $k$-algebra $A$, the category $A$-Mod of all $A$-modules with
morphisms being the module homomorphisms is an abelian category. The derived category $\mathcal{D}(A)$ of $\mathcal{A}$ has objects the complexes of
$A$-modules, and morphisms are obtained from morphisms
in the homotopy category of $A$-Mod by formally inverting all quasi-isomorphisms. 
\end{example}

 A functor $F: \mathcal{A}\rightarrow \mathcal{B}$ between two categories  is said to be an {\it equivalence} of categories if there exists a functor $G:\mathcal{B}\rightarrow\mathcal{A}$ such that the functor $G\circ F$ is isomorphic to $\mathrm{Id}_\mathcal{A} $ and the functor $F\circ G$ is isomorphic to $\mathrm{Id}_\mathcal{B} $. The category $\mathcal{A}$ and $\mathcal{B}$ are said to be equivalence.

\begin{proposition}[See \cite{CW}, Theorem 10.4.8]\label{Prop:CW1111}
Let $A$ be a $k$-algebra and $\mathcal{A}$ be the category of $A$-Mod.
Then the category $\mathcal{D}^{-}(\mathcal{A})$ is equivalent to $\mathcal{K}^{-}(Proj-A)$,
where $\mathcal{K}(Proj-A)$ is the homotopy category of the category of all projective $A$-modules.
\end{proposition}

We next introduce Rickard's definition of derived Morita equivalence which is an equivalence between triangulated categories. To this end we introduce the concept of a triangulated category.
\begin{definition}[Triangle]
Let $\mathcal{A}$ be an additive category together with an self-equivalence
$\Sigma :
 \mathcal{A} \longrightarrow \mathcal{A}$
 (called the translation functor). A sextuple $(A, B, C, u, v, w)$ is called a {\it triangle}
in $\mathcal{A}$ if $A, B, C$ are objects in $\mathcal{A}$,
and $u, v, w$ are morphisms as follows: $u: A \longrightarrow B$,
$v: B \longrightarrow C$, $w: C \longrightarrow \Sigma^{-1}(A)$.
 A triangle is usually written as
$$
\xymatrix{
                & C \ar[dl]_{w}             \\
 A \ar[rr]_{u} & &     B \ar[ul]_{v} .       }
$$
A morphism of triangles $(A, B, C, u, v, w) \longrightarrow (A^{'}, B^{'}, C^{'}, u^{'}, v^{'}, w^{'})$ is a triple $(f,g,h)$ where $f\in Mor_{\mathcal{A}}(A,A^{'}) ,g\in Mor_{\mathcal{A}}(B,B^{'}) $ and $h\in Mor_{\mathcal{A}}(C,C^{'}) $ such that the squares in the diagram
$$
\xymatrix{
  A \ar[d]_{f} \ar[r]^{u}
   & B \ar[d]_{g} \ar[r]^{v}
    & C \ar[d]_{h} \ar[r]^{w}
          & \Sigma^{-1}A  \ar[d]_{\Sigma^{-1}(f)}     \\
  A^{'}  \ar[r]_{u^{'}}
                & B^{'} \ar[r]_{v^{'}}
                & C^{'} \ar[r]_{w^{'}}
                &\Sigma^{-1}A^{'}         }
$$
are commutative.
\end{definition}

\begin{definition}[Triangulated category]
An additive category $\mathcal{A}$ is called a {\it triangulated category} if it is equipped with
a translation functor and a distinguished family of triangles $(A, B, C, u, v, w)$, called the distinguished triangles in $\mathcal{A}$,
subject to the following four axioms:

(TR1) Every morphism $f : A \longrightarrow B$ in $\mathcal{A}$ can be embedded in a distinguished triangle $(A,B,C,f,v,w)$.
If $A = B$ and $C = 0$, then the triangle $(A,A,0,id_{A}, 0, 0)$ is a distinguished triangle. If  $(A, B, C, u, v, w) $ is isomorphic to a distinguished triangle $(A^{'}, B^{'}, C^{'}, u^{'}, v^{'}, w^{'})$, then $(A, B, C, u, v, w) $ is also a distinguished triangle.

(TR2) If $(A, B, C, u, v, w)$ is a distinguished triangle, then $(B, C, \Sigma^{-1}(A), v, w, - \Sigma^{-1}(u))$ and $(\Sigma(C),A, B,- \Sigma (w), u, v)$
are also distinguished triangles.

(TR3) Given two distinguished triangles

$$
\xymatrix{
                & C \ar[dl]_{w}             \\
 A \ar[rr]_{u} & &     B \ar[ul]_{v}        }
\quad\mbox{and}\quad
\xymatrix{
                & C^{'} \ar[dl]_{w^{'}}             \\
 A^{'} \ar[rr]_{u^{'}} & &     B^{'} \ar[ul]_{v^{'}}        }
$$
with morphism $f : A \longrightarrow A^{'}, g: B \longrightarrow B^{'}$ such that $g \circ u = u^{'} \circ f$, there exists a morphism $h: C \longrightarrow C^{'}$ such that $(f, g, h)$ is a morphism of triangles

$$
\xymatrix{
  A \ar[d]_{f} \ar[r]^{u}
   & B \ar[d]_{g} \ar[r]^{v}
    & C \ar[d]_{h} \ar[r]^{w}
          & \Sigma^{-1}A   \ar[d]_{\Sigma^{-1}(f)}     \\
  A^{'}  \ar[r]_{u^{'}}
                & B^{'} \ar[r]_{v^{'}}
                & C^{'} \ar[r]_{w^{'}}
                & \Sigma^{-1}A^{'}.      }
$$

(TR4) Given distinguished triangles
\begin{eqnarray*}
(A, B, C^{'}, u, j, \varphi_{1}), \ (B, C, A^{'}, v, \phi_{1}, i) \text{ and }(A, C, B^{'}, vu, \phi_{2}, \varphi_{2})
\end{eqnarray*}
as in the following octohedron
$$\xymatrixcolsep{5pc}
\xymatrix{
 &  B^{'} \ar@{.>}[dr]^{g} \ar[ddl]^{\varphi_{2}}
  &  \\
    C^{'} \ar@{.>}[ur]^{f} \ar[d]_{\varphi_{1}}
     &
       & A^{'} \ar[ddl]_{i} \ar[ll]_{\Sigma^{-1}(j)i} \\
    A \ar[rr]_{uv} \ar[dr]_{u}
     &
      & C \ar[u]_{\phi_{1}} \ar[uul]^{\phi_{2}}  \\
   &  B \ar[uul]_{j} \ar[ur]_{v}
                                               &   }
$$
there exist morphisms $f: C^{'} \longrightarrow B^{'}$ and $g: B^{'} \longrightarrow A^{'}$ such that
$$
(C^{'}, B^{'}, A^{'}, f, g, \Sigma^{-1}(j)i)
$$
is a distinguished triangle, and the two other faces of the octohedron with $f, g$ as edges are commutative diagrams.
\end{definition}
Actually, a wide class of categories are  triangulated categories.\begin{proposition}[See \cite{RH}, Propositions 3.2 and 4.2]
Let $\mathcal{A}$ be an abelian category. Then
$\mathcal{K}(\mathcal{A})$ and $\mathcal{D}(\mathcal{A})$ are both triangulated categories.
\end{proposition}
\begin{definition}[Triangulated functor]
	A functor $F:\mathcal{A}\rightarrow\mathcal{B}$ between two triangulated categories is called a {\it triangulated functor}, or {\it a functor of triangulated categories} if it commutes with translation functors and preserves the distinguished triangles.
\end{definition}

Now we introduce the definition of a derived functor, following Verdier (\cite{Ve} \S 2 Definition 1.2, Proposition 1.6). 
\begin{definition}[Derived functor]\label{derived functor}
Let $\mathcal{A}$ and $\mathcal{B}$ be two abelian categories. Suppose $F: \mathcal{K}(\mathcal{A}) \longrightarrow \mathcal{K}(\mathcal{B})$ is a functor of triangulated categories. A {\it right derived functor} of $F$ is a functor $RF : \mathcal{D}(\mathcal{A}) \longrightarrow \mathcal{D}(\mathcal{B})$ of triangulated categories, together with a natural transformation $\xi$ from
$$
Q_{\mathcal{B}} \circ F : \mathcal{K}(\mathcal{A}) \longrightarrow \mathcal{K}(\mathcal{B}) \longrightarrow \mathcal{D}(\mathcal{B})
$$
to
$$
RF \circ Q_{\mathcal{A}} : \mathcal{K}(\mathcal{A}) \longrightarrow \mathcal{D}(\mathcal{A}) \longrightarrow \mathcal{D}(\mathcal{B}),
$$
which is universal in the sense that if $G : \mathcal{D}(\mathcal{A}) \longrightarrow \mathcal{D}(\mathcal{B})$ is another functor equipped with a natural transformation
$$
\zeta : Q_{\mathcal{B}} \circ F \Longrightarrow G \circ Q_{\mathcal{A}},
$$
then there exists a natural transformation
$$
\eta : RF \Longrightarrow G
$$
 such that $\zeta_{M} = \eta_{Q_{\mathcal{A}}(M)} \circ \xi_{M}$ in $\mathcal{D}(\mathcal{A})$.

 If $\mathcal{K}^{'} \subset \mathcal{K}(\mathcal{A})$ is a full triangulated subcategory,
 then there is a natural transformation from the right derived functor $R^{'}F$ on $\mathcal{D}^{'}(\mathcal{A})$, the corresponding derived category of $\mathcal{K}^{'}$,
 to the restriction of $RF$ on $\mathcal{D}(\mathcal{A})$.
 We will write $R^{*}F$ for the derived functors of $F$ on
 $ \mathcal{D}^{*}(\mathcal{A})$.

\vspace{3mm}

 Similary, a left derived functor of $F$ is a functor $LF : \mathcal{D}(\mathcal{A}) \longrightarrow \mathcal{D}(\mathcal{B})$ of triangulated categories, together with a natural transformation
 $$
 \xi : LF \circ Q_{\mathcal{A}} \Longrightarrow Q_{\mathcal{B}} \circ F
 $$
  satisfying the dual universal property.
\end{definition}

\begin{proposition}[\cite{RH}, Theorem 5.1]
Let $\mathcal{A}$ and $\mathcal{B}$ be two abelian categories. Suppose there exists
an exact functor $F : \mathcal{K}^{-}(\mathcal{A}) \longrightarrow \mathcal{K}^{-}(\mathcal{B})$
and a triangulated subcategory $\mathcal{L} \subseteq \mathcal{K}^{-}(\mathcal{A})$ such that
\begin{enumerate}
\item[$(1)$] for any object $M \in \mathcal{K}^{-}(\mathcal{A})$, there exists
an object $L \in \mathcal{L}$ and a quasi-isomorphism
$$
i_{M}: L \rightarrow M;
$$
\item[$(2)$] if $I \in \mathcal{L}$ is acyclic, then $F(I)$ is also acyclic.
\end{enumerate}
Then $F$ has a left derived functor
$\mathcal{D}(F):\mathcal{D}^{-}(\mathcal{A}) \longrightarrow \mathcal{D}^{-}(\mathcal{B})$.
Furthermore, for any $M\in \mathcal{K}^{-}(\mathcal{A})$,
$$
\mathcal{D}(F) \circ Q_{\mathcal{A}} (M) = Q_{\mathcal{B}} \circ F (J),
$$
where $J$ is in $\mathcal{L}$ such that there is a quasi-isomorphism $i:J\rightarrow M$,
$Q_{\mathcal{A}}: \mathcal{K}(\mathcal{A}) \longrightarrow \mathcal{D}(\mathcal{A})$
is the localization functor associated to $\mathcal{K}(\mathcal{A})$, and
$Q_{\mathcal{B}}: \mathcal{K}(\mathcal{B}) \longrightarrow \mathcal{D}(\mathcal{B})$
is the localization functor associated to $\mathcal{K}(\mathcal{B})$.
\end{proposition}

Now, let us recall some special derived functors, which will be used in \S3. First, we recall the definition of the enveloping algebra $A^e$ for convenience.
\begin{definition}
Let $A$ be a $k$-algebra, denote by $1_{A} \in A$ the unit in $A$ and by $m_{A} : A \otimes_{k} A \longrightarrow A$ the multiplication.

The opposite algebra $A^{op}$ of $A$ is defined to be the same vector space of $A$ endowed with the multiplication $m_{A^{op}}(a \otimes_{k} b) = m_{A}(b \otimes_{k} a)$, for any $a, b \in A^{op}$, and the unit $1_{A^{op}} = 1_{A} \in A^{op}$.

Let $A$ and $B$ be two $k$-algebras. There is an algebra structure on the vector space $A \otimes_{k} B$ such that its multiplication is given by
$$
m_{A \otimes_{k} B}((a_{1} \otimes b_{1}) \otimes_{k} (a_{2} \otimes b_{2})): = m_{A}(a_{1} \otimes_{k} a_{2}) \otimes_{k} m_{B}(b_{1} \otimes_{k} b_{2}),
$$
for any $a_{1}, a_{2} \in A$ and $b_{1}, b_{2} \in B$, and its unit is $1_{A} \otimes_{k} 1_{B} \in A \otimes_{k} B$, where $1_{A}$ is the unit of $A$, $1_{B}$ is the unit of  $B$, we denote by $A^{e}$ the $k$-algebra $A \otimes_{k} A^{op}$.
\end{definition}

Let $A$, $B$, $C$ be $k$-algebras. Suppose $N$ is a complex of $B \otimes_{k} A^{op}$-modules in $\mathcal{K}^{-}(B \otimes_{k} A^{op})$ and
$M$ is a complex of $A \otimes_{k} C^{op}$-modules in $\mathcal{K}^{-}(A \otimes_{k} C^{op})$.
Then the functor $N \otimes_{A}^{L}(-) : \mathcal{K}^{-}(A) \longrightarrow \mathcal{K}^{-}(B)$ in homotopy categories takes any acyclic object in $\mathcal{K}^{-}(A)$ to an acyclic object in $\mathcal{K}^{-}(B)$.

 Thus, we obtain that the functor $N \otimes_{A}^{L}(-) : \mathcal{K}^{-}(A) \longrightarrow \mathcal{K}^{-}(B)$ can be lifted  to be a derived functor between the associated derived categories.

 Analogously, $\mathrm{Hom}_{\mathcal{K}^{-}(A-Mod)}(M,-) : \mathcal{K}^{-}(A) \longrightarrow \mathcal{K}^{-}(C)$ can also be lifted to be a derived functor, which is denoted by
$$
\mathrm{RHom}_{A}(M, -) : \mathcal{D}^{-}(A) \longrightarrow \mathcal{D}^{-}(C),
$$
 between the associated derived categories. Using this method, we lift more derived functors from homotopy categories (\cite[Charter 1]{RH}, \cite[Charter III.6]{SG}).

We conclude this section by recalling the following concept of derived equivalence.
\begin{definition}[Triangulated equivalence and derived equivalence]
	 Let $\mathcal{A}$ and $\mathcal{B}$ be two triangulated categories. If there is an equivalence functor $F:\mathcal{A}\rightarrow\mathcal{B}$, then $\mathcal{A}$ and $\mathcal{B}$ are said to be {\it triangulated equivalent} or  $\mathcal{A}$ and $\mathcal{B}$ are {\it equivalent as triangulated categories}. Furthermore, if the functor $F$ is a derived functor, $\mathcal{A}$ and $\mathcal{B}$ are said to be {\it derived equivalent}.
\end{definition}
Since the derived functor in Definition \ref{derived functor} is always  a functor of triangulated categories, we have the following corollary:
\begin{corollary}[See \cite{SG} \S 3.6.8]
	A derived equivalence is also a triangulated equivalence.\end{corollary}
	
\section{Derived Morita theory and proof of the Theorem \ref{MainThm}}

In this section, we first collect some facts on derived Morita theory,
initiated by Rickard in \cite{R,JR}. Then we use his results to prove the main theorem. Let $P_A$ denote the category of all finitely generated projective $A$-modules.

\begin{definition}[\cite{R} Theorem 6.4 and Definition 6.5]\label{derived equivalence}
Let $A$ and $B$ be two $k$-algebras. They are called {\it derived Morita equivalent} (or {\it
derived equivalent}
for short) if $A$ and $B$ satisfy one of the following equivalent conditions:
\begin{enumerate}
	\item[$(1)$] $\mathcal{K}^{-}(Proj-A)$ and $\mathcal{K}^{-}(Proj-B)$ are equivalent as triangulated categories;
	
	\item[$(2)$] $\mathcal{D}^{b}(A)$ and $\mathcal{D}^{b}(B)$ are equivalent as triangulated categories;
	
	\item[$(3)$]$\mathcal{K}^{b}(Proj-A)$ and $\mathcal{K}^{b}(Proj-B)$ are equivalent as triangulated categories;
	
	\item[$(4)$] $\mathcal{K}^{b}(P_A)$ and $\mathcal{K}^{b}(P_B)$ are equivalent as triangulated categories;
	
\item[$(5)$]$B$ is isomorphic to $\mathrm{End}(T)$, where $T$ is an object in $\mathcal{K}^b(P_A)$ satisfying

(a) $\mathrm{Hom}(T,T[i])=0$ for $i\neq 0$ and
 
(b) the category $add(T)$ of direct summands of finite direct sums of copies of $T$, generates $\mathcal{K}^b(P_A)$ as a triangulated category.
\end{enumerate}
\end{definition}

Let us recall the definition of a Calabi-Yau algebra.

\begin{definition}[Ginzburg \cite{G}]\label{Def:CYalgebra}
A $k$-algebra $A$ is called $d$-Calabi-Yau if $A$, viewed as an $A^e$-module,
has a bounded resolution of finitely generated projective $A^{e}$-modules,
and there is an isomorphism $\mathrm{RHom}_{A^{e}} (A, A\otimes_k A) \cong A[-d]$ in $\mathcal{D}(A^{e})$.
\end{definition}

In above definition, we have used the $A^e$-module structure on
$A \otimes_{k} A$ which is given by
$$
(a_{1} \otimes_{k} b_{1})\circ (a_{2} \otimes_{k} b_{2}) = a_{1}a_{2} \otimes_{k} b_{2}b_{1},
$$
for any $a_{1} \otimes_{k} b_{1} \in A^{e}$ and $a_{2} \otimes_{k} b_{2} \in A \otimes_{k} A$.
$A\otimes_k A$ also has a right $(A^{e})^{op}$-module (or equivalently left $A^e$-module) structure, which is given by
$$
(a_{1} \otimes_{k} b_{1})\ast (a_{2} \otimes_{k} b_{2}) = a_{1}b_{2} \otimes_{k} a_{2}b_{1},
$$
for any $a_{2} \otimes_{k} b_{2} \in (A^{e})^{op}$ and $a_{1} \otimes_{k} b_{1} \in A \otimes_{k} A$. The two structures are compatible so that $A \otimes_{k} A$ has an $A^{e} \otimes_{k} ((A^{e})^{op})$-module structure,
and the $A^e$-module structure on $\mathrm{R}\mathrm{Hom}_{A^e}(A,  A\otimes_k A)$
is induced by this second module structure on $A\otimes_k A$. 

 Now let us turn to the proof of our main result.
We divide our proof into two steps. First, we verify the following lemma:

\begin{lemma}\label{Lemma:1}
Let $A$ and $B$ be two $k$-algebras. Suppose they are derived equivalent. If there is an isomorphism
$$
\mathrm{RHom}_{A^{e}} (A,  A\otimes_k A) \cong A[-d]
$$
in $\mathcal{D}^b(A^{e})$, then we have
$$
\mathrm{RHom}_{B^{e}} (B,  B\otimes_k B) \cong B[-d]
$$
in $\mathcal{D}^b(B^{e})$.
\end{lemma}

The main techniques to prove this lemma are two propositions proved by Rickard in \cite{JR}.
\begin{proposition}[\cite{BK} Theorem 6.4; \cite{JR} Corollary 3.5 and Definition 4.2]\label{JR,BK}
Let $A$ and $B$ be two $k$-algebras. If they are derived equivalent, then there is a bounded complex $Q$ of $A \otimes_{k} B^{op}$-modules, a bounded complex $P$ of $B \otimes_{k} A^{op}$-modules which is isomorphic to $\mathrm{RHom}_{A}(Q, A)$ in $\mathcal{D}^b(B \otimes_{k} A^{op})$, an isomorphism in $\mathcal{D}^b(A^e)$:
$$
u: Q \bigotimes^\textsl{L}_\textrm{B} P \cong A,
$$
 an isomorphism in $\mathcal{D}^b(B^e)$:
 $$
 v: P \bigotimes^\textsl{L}_A Q \cong B,
 $$
and the composition is also a derived equivalence:
 $$
 P \otimes_{A}^{L} (-) \otimes_{A}^{L} Q: D^b(A^e) \longrightarrow D^b(B^e).
 $$
\end{proposition}

\begin{proposition}[\cite{JR} Theorem 4.4]\label{Prop:JR}
Let $A_{i}$ and $B_{i}$ be $k$-algebras ($i = 0, 1, 2$),
such that they are derived equivalent.
Suppose the derived equivalences are given by
$$Y_{i} \otimes_{A_{i}}^{L} (-)  : \mathcal{D}^b(A_{i}) \rightarrow \mathcal{D}^b(B_{i}),$$
where
$$
Y_{i} \cong \mathrm{RHom}_{A_{i}}(X_{i}, A_{i})
$$
 and
 $$
 X_{i} \in \mathcal{D}^{b}(A_{i} \otimes_{k} B_{i}^{op})
 $$
  $(i = 0, 1, 2)$ such that
  $
  Y_{i} \otimes_{A_{i}}^{L} X_{i} \cong B_{i}
  $
   in $\mathcal{D}^b(B_{i}^{e})$ and
   $
   X_{i} \otimes_{B_{i}}^{L} Y_{i} \cong A_{i}
   $
   in $\mathcal{D}^b(A_{i}^{e})$.

Then we have derived equivalences
$$Y_{i} \otimes_{A_{i}}^{L} (-) \otimes^{L}_{A_{j}}X_{j} : \mathcal{D}^{-}(A_{i} \otimes_{k} A_{j}^{op})
\longrightarrow \mathcal{D}^{-}(B_{i} \otimes B_{j}^{op}),$$ $(i, j = 0, 1, 2)$ and two commutative diagrams of functors:
$$\xymatrixcolsep{5pc}
\xymatrix {
     \mathcal{D}^{-}(A_{0} \otimes_{k} A_{1}^{op}) \times \mathcal{D}^{-}(A_{1} \otimes_{k} A_{2}^{op}) \ar[d]_{(Y_{0} \otimes_{A_{0}}^{L} (-) \otimes_{A_{1}}^{L}X_{1}) \times (Y_{1} \otimes_{A_{1}}^{L} (-) \otimes_{A_{2}}^{L}X_{2})}
     \ar[r]^{\quad  \quad \quad (-) \otimes^{L}_{A_{1}} (-)} & \mathcal{D}^{-}(A_{0} \otimes_{k} A_{2}^{op}) \ar[d]^{Y_{0} \otimes_{A_{0}}^{L} (-) \otimes^{L}_{A_{2}}X_{2}}  \\
     \mathcal{D}^{-}(B_{0} \otimes_{k} B_{1}^{op}) \times \mathcal{D}^{-}(B_{1} \otimes_{k} B_{2}^{op})   \ar[r]_{ \quad \quad \quad (-) \otimes^{L}_{B_{1}} (-)} & \mathcal{D}^{-}(B_{0} \otimes_{k} B_{2}^{op})
           }
$$
and
$$
\xymatrixcolsep{5pc}
\xymatrix {
     \mathcal{D}^{b}(A_{2} \otimes_{k} A_{1}^{op})^{op} \times \mathcal{D}^{b}(A_{2} \otimes_{k} A_{0}^{op}) \ar[d]_{(Y_{2} \otimes_{A_{2}}^{L} (-) \otimes_{A_{1}}^{L}X_{1}) \times (Y_{2} \otimes_{A_{2}}^{L} (-) \otimes_{A_{0}}^{L}X_{0})}
     \ar[r]^{\quad \quad \quad \quad  \mathrm{RHom}_{A_{2}}(-, -)} & \mathcal{D}^{+}(A_{1} \otimes_{k} A_{0}^{op}) \ar[d]^{Y_{1} \otimes_{A_{1}}^{L} (-) \otimes^{L}_{A_{2}}X_{2}}  \\
     \mathcal{D}^{b}(B_{2} \otimes_{k} B_{1}^{op})^{op} \times \mathcal{D}^{b}(B_{2} \otimes_{k} B_{0}^{op})   \ar[r]_{\quad \quad \quad \quad \quad \mathrm{RHom}_{B_{2}}(-, -)} & \mathcal{D}^{+}(B_{1} \otimes_{k} B_{0}^{op}).
           }
$$
\end{proposition}
We are now ready to prove Lemma \ref{Lemma:1}.
\begin{proof}[Proof of Lemma \ref{Lemma:1}]
Since $A$ is derived equivalent to $B$, by Proposition \ref{JR,BK},
there is a bounded complex of $A \otimes_{k} B^{op}$-modules $Q$,
a bounded complex of $B \otimes_{k} A^{op}$-modules $P$  such that $P \cong \mathrm{RHom}_{A}(Q, A)$ in $\mathcal{D}^b(B \otimes_{k} A^{op})$, $Q \otimes_{B}^{L} P \cong A$ in $\mathcal{D}^b(A^{e})$, and $P \otimes_{A}^{L} Q \cong B$ in $\mathcal{D}^b(B^{e})$,
then the functor
 $$
 P \otimes_{A}^{L} (-) \otimes_{A}^{L} Q: \mathcal{D}^b(A^{e}) \longrightarrow \mathcal{D}^b(B^{e})
 $$
 is a derived equivalence functor and  so is the functor
 $$
 P \otimes_{A}^{L} (-) \otimes_{A}^{L} Q: \mathcal{D}^b((A^{e})^{op}) \longrightarrow \mathcal{D}^b((B^{e})^{op}).
 $$
Suppose that $\mathrm{RHom}_{A^{e}} (A, A\otimes A) \cong A[-d]$ in $\mathcal{D}^b(A^{e})$. Using
the above derived equivalences and Proposition
\ref{Prop:JR}, let $A_{2} = A^{e}$, $B_{2} = B^{e}$, $A_{1} = B_{1} = k$, $A_{0} = (A^{e})^{op}$, $B_{0} = (B^{e})^{op}$, $Y_{1} = k$, $ X_{1}=k$, $Y_{2} = P\otimes Q^{op}$, $X_2=Q\otimes P^{op}$ and $X_0=Q^{op}\otimes P$; we have the commutative diagram of functors $$
\xymatrixcolsep{5pc}
\xymatrix {
     \mathcal{D}^{b}(A^{e}) \times \mathcal{D}^{b}(A^{e} \otimes_{k} A^{e}) \ar[d]_{(P\otimes Q^{op} \otimes_{A^{e}}^{L} (-)) \times (P\otimes Q^{op} \otimes_{A^e}^{L} (-) \otimes_{(A^{e})^{op}}^{L}Q^{op}\otimes P)}
     \ar[r]^{\quad \quad \quad \quad  \mathrm{RHom}_{A^{e}}(-, -)} & \mathcal{D}^{+}(A^{e}) \ar[d]^{ (-) \otimes^{L}_{A^e}Q\otimes P^{op}}  \\
     \mathcal{D}^{b}(B^{e}) \times \mathcal{D}^{b}(B^{e} \otimes_{k} B^{e})   \ar[r]_{\quad \quad \quad \quad \quad \mathrm{RHom}_{B^{e}}(-, -)} & \mathcal{D}^{+}(B^{e}).
           }
$$
Now, taking $A\in\mathcal{D}^{b}(A^{e})$ and $A\otimes_kA\in\mathcal{D}^{b}(A^{e} \otimes_{k} A^{e})$, we obtain \[\mathrm{RHom}_{A^{e}} (A, A\otimes_k A )\in\mathcal{D}^{b}(A^{e}), \]
\[ P\otimes Q^{op} \otimes_{A^{e}}^{L}A\cong P \otimes_{A}^{L} A \otimes_{A}^{L} Q \in \mathcal{D}^{b}(B^{e}),\]
and
\[
P\otimes Q^{op} \otimes_{A^e}^{L}A\otimes_k A  \otimes_{(A^{e})^{op}}^{L}Q^{op}\otimes P\cong P \otimes_{A}^{L} A \otimes_{A}^{L} Q \otimes_{k} P \otimes_{A}^{L} A \otimes_{A}^{L} Q\in\mathcal{D}^{b}(B^{e} \otimes_{k} B^{e}) .
\]

 Since the diagram commutes, we have: \begin{eqnarray*}
	&& \mathrm{RHom}_{B^{e}} (P \otimes_{A}^{L} A \otimes_{A}^{L} Q, P \otimes_{A}^{L} A \otimes_{A}^{L} Q \otimes_{k} P \otimes_{A}^{L} A \otimes_{A}^{L} Q)
	\\ &\cong &\mathrm{RHom}_{A^{e}} (A, A\otimes_k A )\otimes^{L}_{A^e}(Q\otimes P^{op})
	\\ &= & P \otimes_{A}^{L} \mathrm{RHom}_{A^{e}} (A, \otimes_k A) \otimes_{A}^{L} Q
	\\&\cong & P \otimes_{A}^{L} A[-d] \otimes_{A}^{L} Q.
\end{eqnarray*}
Note that
$A\cong Q \otimes_{B}^{L} P $ in $\mathcal{D}^b(A^{e})$ and $P \otimes_{A}^{L} Q\cong B $ in $\mathcal{D}^b(B^{e})$; the above identity gives
 $$
 \mathrm{RHom}_{B^{e}} (B, B \otimes_{k} B) \cong B[-d]
 $$
  in $\mathcal{D}^b(B^{e})$ (by Proposition \ref{JR,BK}).
\end{proof}
Now we proceed to prove the following lemma.
\begin{lemma}\label{Lemma:2}
Suppose $A$ and $B$ are two derived equivalent $k$-algebras. If $A$ viewed as an $A^e$-module has a bounded resolution of finitely generated projective $A^{e}$-modules, then so is $B$.
\end{lemma}




\begin{proof}
Suppose there is an isomorphism $\pi:X\rightarrow A$ in $\mathcal{D}^b(A^{e})$, where $X \in \mathcal{K}^{b}(P_{A^{e}})$.
By Propositions \ref{JR,BK}, we have the derived equivalence 
$$
P \otimes_{A}^{L} (-) \otimes_{A}^{L} Q: \mathcal{D}^b(A^{e}) \longrightarrow \mathcal{D}^b(B^{e}),
$$
and by Definition \ref{derived equivalence},
 $$
 \widetilde{F}: \mathcal{K}^{b}(P_{A^{e}}) \longrightarrow \mathcal{K}^{b}(P_{B^{e}}),
 $$
 is also a derived equivalence. Since an equivalence functor is always fully faithful and the morphisms $\pi$ in $ \mathcal{D}^b(A^{e})$ is an isomorphism, $\widetilde{F}(\pi):\widetilde{F}(X)\rightarrow\widetilde{F}(A)$ is an isomorphism in $\mathcal{D}^b(B^{e})$, that is to say, $\widetilde{F}(X) \in \mathcal{K}^{b}(P_{B^{e}})$ is quasi-isomorphic to $\widetilde{F}(A) = P \otimes_{A}^{L} A \otimes_{A}^{L} Q \cong B$  in $\mathcal{K}^b(B^{e})$.
\end{proof}

\begin{proof}[Proof of Theorem \ref{MainThm}]
Suppose $A$ is a $d$-Calabi-Yau algebra. Then
Lemmas
\ref{Lemma:1} and \ref{Lemma:2} exactly imply that $B$ satisfies the conditions in
Definition \ref{Def:CYalgebra} and is therefore a $d$-Calabi-Yau algebra. Thus, we have proved the theorem.
\end{proof}

\end{document}